\documentclass[letterpaper]{article}

\newcommand{\mytitle}{Derived categories of torsors for abelian schemes}

\title{\mytitle}
\author{Benjamin Antieau\footnote{Benjamin Antieau was supported by NSF Grant DMS-1358832.},
Daniel Krashen\footnote{Daniel Krashen was supported by NSF Grant DMS-1151252.}, and Matthew Ward}

\usepackage[pdfstartview=FitH,
            colorlinks,
            linkcolor=reference,
            citecolor=citation,
            urlcolor=e-mail,
            backref]{hyperref}

\usepackage{amsmath}
\usepackage{amscd}
\usepackage{amsbsy}
\usepackage{amssymb}
\usepackage{verbatim}
\usepackage{eufrak}
\usepackage{microtype}
\usepackage{hyperref}
\usepackage{mathrsfs}
\usepackage{amsthm}
\usepackage[all,cmtip]{xy}
\usepackage{tikz}
\usetikzlibrary{matrix,arrows}
\usepackage[abbrev,lite]{amsrefs}

\usepackage{mathpazo}

\usepackage{color}
\definecolor{todo}{rgb}{1,0,0}
\definecolor{conditional}{rgb}{0,1,0}
\definecolor{e-mail}{rgb}{0,.40,.80}
\definecolor{reference}{rgb}{.20,.60,.22}
\definecolor{mrnumber}{rgb}{.80,.40,0}
\definecolor{citation}{rgb}{0,.40,.80}

\DeclareMathAlphabet{\mathpzc}{OT1}{pzc}{m}{it}
\usepackage{dsfont}

\DeclareMathOperator{\Ho}{Ho}

\DeclareMathOperator{\id}{id}

\DeclareMathOperator{\op}{op}

\newcommand{\rwe}{\tilde{\rightarrow}}

\newcommand{\we}{\simeq}
\newcommand{\iso}{\cong}

\newcommand{\Gm}{\mathbb{G}_{m}}

\newcommand{\Gal}{\mathrm{Gal}}

\newcommand{\tors}{\mathrm{tors}}

\newcommand{\qc}{\mathrm{qc}}

\newcommand{\dg}{\mathrm{dg}}

\newcommand{\THHB}{\mathbf{HH}}

\newcommand{\Mod}{\mathrm{Mod}}
\newcommand{\Perf}{\mathrm{Perf}}

\DeclareMathOperator{\Ch}{Ch}

\DeclareMathOperator{\Spec}{Spec}

\DeclareMathOperator{\SL}{SL}

\DeclareMathOperator{\NS}{NS}

\DeclareMathOperator{\Hoh}{H}

\newcommand{\Hom}{\mathrm{Hom}}
\newcommand{\Map}{\mathrm{Map}} % enriched maps

\DeclareMathOperator{\End}{End}
\DeclareMathOperator{\Aut}{Aut}

\newcommand{\QC}{\mathrm{QC}}

\newcommand{\ShBr}{\mathbf{Br}}

\newcommand{\ShAut}{\mathbf{Aut}}
\newcommand{\StQC}{\mathscr{QC}}
\newcommand{\StQ}{\mathscr{Q}}

\newcommand{\et}{\mathrm{\acute{e}t}}

\DeclareMathOperator{\Br}{Br}

\newcommand{\Lrm}{\mathrm{L}}

\newcommand{\Drm}{\mathrm{D}}

\newcommand{\Qrm}{\mathrm{Q}}
\newcommand{\Urm}{\mathrm{U}}

\newcommand{\Ubf}{\mathbf{U}}
\newcommand{\Vbf}{\mathbf{V}}

\newcommand{\Oscr}{\mathcal{O}}

\newcommand{\Cscr}{\mathscr{C}} % general R-linear categories
\newcommand{\Dscr}{\mathscr{D}}

\newcommand{\CC}{\mathbb{C}}
\newcommand{\RR}{\mathbb{R}}
\newcommand{\QQ}{\mathbb{Q}}
\newcommand{\ZZ}{\mathbb{Z}}

\newcommand{\FF}{\mathbb{F}}
\newcommand{\PP}{\mathbb{P}}

\theoremstyle{plain}
\newtheorem{theorem}{Theorem}[section]
\newtheorem{lemma}[theorem]{Lemma}
\newtheorem{proposition}[theorem]{Proposition}

\newtheorem{corollary}[theorem]{Corollary}

\theoremstyle{definition}
\newtheorem{definition}[theorem]{Definition}

\newtheorem{example}[theorem]{Example}

\newtheorem{remark}[theorem]{Remark}

\let\oldmarginpar\marginpar
\renewcommand\marginpar[1]{\-\oldmarginpar[\raggedleft\footnotesize #1]%
{\raggedright\footnotesize #1}}

\begin{document}

\maketitle

\begin{abstract}
    \noindent
    In the first part of our paper, we show that there exist non-isomorphic derived
    equivalent genus $1$ curves, and correspondingly there exist non-isomorphic moduli
    spaces of stable vector bundles on genus $1$ curves in general.
    Neither occurs over an algebraically closed
    field. We give necessary and sufficient conditions for two genus $1$ curves to be
    derived equivalent, and we go on to study when two principal homogeneous spaces for an abelian
    variety have equivalent derived categories. We apply our results to study
    twisted derived equivalences of the form 
    $\Drm^b(J,\alpha)\we\Drm^b(J,\beta)$, when $J$ is an elliptic fibration, giving a
    partial answer to a question of C\u{a}ld\u{a}raru.

    \paragraph{Key Words.}
    Derived equivalence, genus $1$ curves, elliptic $3$-folds, and Brauer groups.

    \paragraph{Mathematics Subject Classification 2010.}
    Primary:
    \href{http://www.ams.org/mathscinet/msc/msc2010.html?t=14Fxx&btn=Current}{14F05},
    \href{http://www.ams.org/mathscinet/msc/msc2010.html?t=18Exx&btn=Current}{18E30}.
    Secondary:
    \href{http://www.ams.org/mathscinet/msc/msc2010.html?t=14Dxx&btn=Current}{14D20},
    \href{http://www.ams.org/mathscinet/msc/msc2010.html?t=14Fxx&btn=Current}{14F22}.

\end{abstract}

\section{Introduction}

Two smooth projective varieties $X$ and $Y$ over a field $k$ are derived equivalent if there
is a $k$-linear triangulated equivalence $\Drm^b(X)\we\Drm^b(Y)$. The first example of this
phenomenon was
discovered by Mukai~\cite{mukai}, who showed that $\Drm^b(A)\we\Drm^b(\hat{A})$ for any abelian variety
$A$, where $\hat{A}$ denotes the dual abelian variety.
Finding non-isomorphic derived equivalent varieties has since become a central problem
in algebraic geometry, closely bound up with homological mirror symmetry and the study of
moduli spaces of vector bundles. For an overview, see the book~\cite{huybrechts}
of Huybrechts.

Letting $P$ denote the Poincar\'e line bundle on $A\times_k\hat{A}$, and denoting by $p$ and
$q$ the projections from the product to $A$ and $\hat{A}$, respectively, Mukai's equivalence
is the functor $\Phi_P:\Drm^b(A)\rightarrow\Drm^b(\hat{A})$, where, for a complex
$F\in\Drm^b(A)$,
\begin{equation*}
    \Phi_P(F)=Rp_*(q^*F\otimes^\Lrm P).
\end{equation*}
Orlov~\cite{orlov-equivalences} showed that any
equivalence $\Drm^b(X)\we\Drm^b(Y)$ is of the form $\Phi_P$ for some complex
$P\in\Drm^b(X\times_k Y)$ when $X$ and $Y$ are smooth projective $k$-varieties. We will
refer to these as Fourier-Mukai equivalences and to $P$ as the kernel.

As a consequence of the existence of kernels,
derived equivalent varieties have the same dimension, the same $K$-theory spectrum, the same
Hochschild homology spectrum, and, by Popa and Schnell~\cite{popa-schnell}, closely related Jacobian varieties.
Moreover, their canonical classes are
simultaneously torsion (of the same order), ample, or anti-ample. And yet, derived
equivalence is coarse enough to allow many interesting examples.

If $\Drm^b(X)\we\Drm^b(Y)$ and the canonical line bundle $\omega_X$ is either ample
or anti-ample, then $X\iso Y$ by a theorem of Bondal and Orlov~\cite{bondal-orlov}. If $X$
and $Y$ are genus $1$ curves over an algebraically closed field (elliptic curves after
choosing basepoints), it is known that $\Drm^b(X)\we\Drm^b(Y)$ implies $X\iso Y$. See for
instance~\cite{huybrechts}*{Corollary~5.46}. Thus, over a separably closed field,
derived equivalent curves are isomorphic.

We are interested in two questions in this paper: first when are two genus $1$ curves derived
equivalent? This is the only remaining
unknown situation for curves. Second, when are two homogeneous spaces for a fixed abelian
variety derived equivalent?
It turns out that the first problem is a special case of the second.

A closely related problem in the curve case is to understand the moduli spaces of
stable vector bundles on genus $1$ curves.
The underlying reason that derived equivalent elliptic curves are isomorphic
is the result of Atiyah~\cite{atiyah}, which shows
that the moduli spaces $J_E(r,d)$ of rank $r$ and degree $d$ stable vector bundles
on an elliptic curve $E$ are themselves isomorphic to $E$. Indeed, given an equivalence
$\Drm^b(F)\we\Drm^b(E)$ of two genus $1$ curves over an algebraically closed field, it is
possible to show that the equivalence sends the structure sheaves of points of $F$ to
semi-stable vector bundles on $E$. This realizes $F$ as a moduli space $J_E(r,d)$ for
some $r$ and $d$. Pumpl\"un~\cite{pumplun} made
some extensions of Atiyah's result in the case of a genus $1$ curve without a point, but did
not obtain a full classification. In particular, it has remained open whether or not every
moduli space $J_X(r,d)$ of vector bundles on a genus $1$ curve is isomorphic to $X$.

We show in this paper that there are non-isomorphic derived equivalent genus $1$ curves, and
that they lead to examples of non-isomorphic moduli spaces of stable vector bundles.

\begin{theorem}\label{thm:main}
    Let $k$ be a field, and let $X$ and $Y$ be two smooth projective genus $1$ curves over $k$.
    Suppose that $\Drm^b(X)\we\Drm^b(Y)$ as $k$-linear triangulated categories. Then, $X$
    and $Y$ have the same Jacobian. If $X$ and $Y$ are principal homogeneous spaces
    for an elliptic curve $E$, then $\Drm^b(X)\we\Drm^b(Y)$ if and only if
    there exists an automorphism $\phi$ of $E$ such that $X=\phi_*dY$ in $\Hoh^1(k,E)$ for some $d$ prime to the order of $Y$.
\end{theorem}

Here and in the rest of the paper, $\Hoh^i(k,-)$ denotes the \'etale cohomology group (or
pointed set) $\Hoh^i_{\et}(\Spec k,-)$.

\begin{corollary}
    There exist non-isomorphic derived equivalent genus $1$ curves.
\end{corollary}

In contrast to work of Atiyah~\cite{atiyah} when $X$ has a $k$-point, we find the following
corollary in the course of our proof.

\begin{corollary}
    There exists a genus $1$ curve $X$ and coprime integers $r\neq 0$ and $d\neq 0$ such that
    $J_X(r,d)$ and $X$ are not isomorphic.
\end{corollary}

Our results on derived equivalences of principal homogeneous spaces in higher dimensions
rely on the twisted Brauer space, a tool for studying
\'etale twisted forms of derived categories introduced by the first named author
in~\cite{antieau-tbs}. We give less complete results in higher
dimensions, establishing a necessary condition for two principal homogeneous spaces for an
abelian variety to be derived equivalent. As the conclusion in the general case is somewhat complicated, we
extract one particular case to highlight here.

\begin{theorem}\label{thm:introhomogeneous}
    Suppose that $X$ and $Y$ are principal homogeneous spaces for an abelian variety $A$
    over a field $k$ such that $\End(A_{k^s})\iso\ZZ$ and $\NS(A_{k^s})$ is the constant
    Galois module $\ZZ$. Then, if $\Drm^b(X)\we\Drm^b(Y)$, $X$ and $Y$ generate the same
    subgroup of $\Hoh^1(k,A)$. Conversely, if $X$ and $Y$ generate the same subgroup, then
    $\Drm^b(X)\we\Drm^b(Y,\beta)$ for some $\beta\in\Br(k)$.
\end{theorem}

The theorem allows one to study derived equivalences of abelian fibrations, and we do so in
the elliptic case, to give a partial answer to a question of C\u{a}ld\u{a}raru.

\begin{theorem}
    Suppose that $p:J\rightarrow S$ is a smooth Jacobian elliptic fibration, where $S$ is a connected
    regular noetherian scheme with characteristic $0$ field of fractions, and suppose that the geometric
    generic fiber $J_\eta$ is geometrically not CM. If there is an equivalence
    $F:\Drm^b(J,\alpha)\we\Drm^b(J,\beta)$ compatible with $p$ for some
    $\alpha,\beta\in\Br(J)$, then $\alpha$ and $\beta$ generate the same cyclic subgroup of
    $\Br(J)/\Br(S)$.
\end{theorem}

Our work here is in addition closely related to the work of
Bridgeland~\cite{bridgeland-elliptic} and
Bridgeland-Maciocia~\cite{bridgeland-maciocia,bridgeland-maciocia-fm} on
derived equivalences of surfaces. In the case where $S$ in the theorem is a smooth curve
over the complex numbers, our theorem easily gives the converse
of~\cite{bridgeland-elliptic}*{Theorem 1.2}, since $\Br(S)=0$ in this case. This converse also occurs
in~\cite{bridgeland-maciocia}*{Section 4}. In particular, they show that derived
equivalences $\Drm^b(X)\we\Drm^b(Y)$ of smooth projective elliptic
surfaces of non-zero Kodaira dimension are all compatible
with appropriate elliptic structures $p:X\rightarrow S$ and $q:Y\rightarrow S$.

Here is a brief description of the contents of our paper. Our results on genus $1$ curves
are proved in Section~\ref{sec:g1c}. In Section~\ref{sec:dg}, we give background on dg
categories, needed for the twisted Brauer space. The twisted
Brauer space is introduced in Section~\ref{sec:tbs}, and it is used in
Section~\ref{sec:derivedhomogeneous} to give restrictions on when two principal
homogeneous spaces can be derived equivalent. An example of an application to elliptic fibrations is
in Section~\ref{sec:elliptic}.

This paper originated at the AIM workshop \emph{Brauer groups and obstruction problems:
moduli spaces and arithmetic} in February 2013. We thank AIM for its hospitality and for
providing a stimulating environment as well as Max Lieblich and Olivier Wittenberg for
informative conversations.

\section[Derived equivalences of genus $1$ curves]{Derived equivalences of genus $1$ curves}\label{sec:g1c}

We describe when two genus $1$ curves are derived equivalent.

\subsection{Elliptic Curves}

Fix a field $k$.
An elliptic curve $E$ is defined to be a smooth geometrically connected genus $1$ curve over $k$ with a distinguished rational point $e: \Spec k\to E$.
The derived category $\Drm^b(E)$ is the bounded derived category of coherent sheaves on $E$. 

By the theorem of Bondal and Orlov, \cite{huybrechts}*{Theorem 5.14},
any derived equivalence $\Drm^b(X)\rightarrow\Drm^b(Y)$ (always triangulated and $k$-linear)
between two smooth projective
varieties is isomorphic to the Fourier-Mukai transform $\Phi=\Phi_P$ associated to a complex
$P$ in $\Drm^b(X\times_k Y)$. Moreover, the complex $P$ is
unique up to quasi-isomorphism. Unless it is explicitly needed, we suppress the name of the kernel.
Given an object $F$ of $\Drm^b(X)$ we write $\Phi^i(F)$ for the $i$th cohomology sheaf
$\mathcal{H}^i(\Phi(F))$.
We use the notation $\mathcal{O}_x$ for the skyscraper sheaf at $x$ with value the residue field $k(x)$.

\begin{lemma}
    Suppose that $p:X\rightarrow Y$ is a morphism of noetherian schemes and that $P$ is a
    perfect complex on $X$. If $y\in Y$ is a closed point such that
    $\mathcal{H}^i(P_y)=0$ for $i\neq 0$, then there is a neighborhood
    $U\subseteq Y$ containing $y$ such that $P_U$ is (quasi-isomorphic to) a coherent
    $\Oscr_{U\times_Y X}$-module flat over $U$.
\end{lemma}

\begin{proof}
    Taking $Z=\Spec\mathcal{O}_{Y,y}$, we can
    apply~\cite{huybrechts}*{Lemma~3.31} over this local ring (since there is only one
    closed point) to find that
    $P_Z$ is quasi-isomorphic to a coherent sheaf flat over
    $Z$. Denote by $j$ the map $Z\times_Y X\rightarrow X$. We have found an equivalence
    $j^*P\rightarrow F$, where $F$ is a coherent sheaf on $Z\times_Y X$ that is flat over
    $Z$. As $j$ is flat, the higher derived functors of $j^*$ vanish, from which it follows
    that $j^*\mathcal{H}^i(P)\iso\mathcal{H}^{i}(j^*P)$. Because $P$ is perfect, there are
    only finitely many non-zero cohomology sheaves, and $j^*\mathcal{H}^i(P)=0$ if $i\neq
    0$. Hence, there is a neighborhood $U$ of $y$ on which $\mathcal{H}^i(P)$ vanishes for
    each $i\neq 0$. The lemma now follows from a second application
    of~\cite{huybrechts}*{Lemma~3.31}.
\end{proof}

The following theorem is well-known to the experts and appears first (in the algebraically
closed case) as far as we can tell
in Bridgeland's thesis~\cite{bridgeland-thesis}. A proof in the case that the base field
is $\CC$ appears as~\cite{huybrechts}*{Corollary~5.46}. We present a slightly different
proof.

\begin{theorem}\label{thm:elliptic}
    If $E$ and $F$ are elliptic curves over a field $k$ such that $\Phi:\Drm^b(E)\simeq\Drm^b(F)$
    as $k$-linear triangulated categories, then there is an isomorphism of $k$-schemes $E\iso F$.
\end{theorem}

\begin{proof}
    Let $x$ be a closed point of $E$.
    We observe that there is a unique value of $i$ such that $\Phi^i(\mathcal{O}_x)\neq
    0$. Indeed, since $F$ is an elliptic curve, the abelian category of quasi-coherent
    sheaves on $F$ is hereditary. In particular,
    \begin{equation*}
        \Phi(\mathcal{O}_x)\iso\bigoplus_i\Phi^i(\mathcal{O}_x)
    \end{equation*}
    in $\Drm^b(F)$.
    Since $\mathcal{O}_x$ is a simple sheaf, it follows that $\Phi(\mathcal{O}_x)$ is
    simple and hence that $\Phi^i(\Oscr_x)$ is non-zero for a unique value of $i$.

    By the lemma, it follows that this $i$ does not depend on $x$, and moreover,
    $\Phi\we\Phi_{P[i]}$ where $P$ is a coherent sheaf on $E\times_k F$ flat over
    $E$. If the $\Phi^i(\mathcal{O}_x)$ are torsion-free, then they are stable vector
    bundles on $F$ since they are simple vector bundles on a genus $1$ curve.
    Define $r=rk(\Phi^i(\mathcal{O}_x))$ and $d=deg(\Phi^i(\mathcal{O}_x))$.
    In this notation, the kernel realizes $E$ as the moduli space $J_F (r,d)$ of
    semi-stable vector bundles of rank $r$ and degree $d$ on $F$. The geometric points of
    $E$ correspond to simple semi-stable vector bundles on $F_{k^s}$. These are
    therefore stable bundles, so that $r$ and $d$ are relatively prime (see
    Bridgeland~\cite{bridgeland-thesis}*{Lemma 6.3.6}).
    The determinant map $J_F(r,d)\rightarrow J_F(1,d)$ is hence an isomorphism by
    by~\cite{atiyah}*{Theorem~7}. But, tensoring with $\mathcal{O}(d\cdot e)$ where $e:\Spec
    k\rightarrow F$ defines an isomorphism $Jac(F)=J_F(1,0)\rightarrow J_F(1,d)$, so that
    $E\iso F$.
    When the sheaves $\Phi^i(\mathcal{O}_x)$ have torsion,
    they are supported at points, and since they are simple they must correspond to skyscraper sheaves of single points.
    Thus, the kernel $P$ of $\Phi$ is a translation of a line bundle on the graph of an isomorphism $E\iso F$
    by~\cite{huybrechts}*{Corollary~5.23}.
\end{proof}

\subsection{Genus 1 Curves}

The main theorem of this section
is an analogue of the main result of the paper in the case of curves.
It fully settles when two curves are derived equivalent over any base field.
Before stating the theorem we can simplify the problem using the following lemmas.

\begin{lemma}\label{lem:jy1d}
    If $X$ and $Y$ are genus $1$ curves over a field $k$ such that $\Drm^b(X)\simeq\Drm^b(Y)$,
    then $X\iso J_Y(1,d)$, the fine moduli space of degree $d$ line bundles on $Y$.
\end{lemma}

\begin{proof}
    Fix an equivalence $\Phi:\Drm^b(X)\we\Drm^b(Y)$. As in the proof of
    Theorem~\ref{thm:elliptic}, $X\iso J_Y(r,d)$ for some $r$ and $d$.
    We may assume that $r\geq 1$ or else they are skyscraper sheaves and hence $X\iso Y$
    by~\cite{huybrechts}*{Corollary~5.23}.
    We may also assume that $r$ and $d$ are coprime by~\cite{bridgeland-thesis}*{Theorem
    6.4.3}. Indeed, by base-change, we obtain an equivalence
    $\Drm^b(X_{k^s})\we\Drm^b(Y_{k^s})$ from
    Propositions~\ref{prop:triadg} and~\ref{prop:localequivalence}. Now, since $r$ and $d$
    are coprime, the determinant map $J_Y(r,d)\rightarrow J_Y(1,d)$ is geometrically an
    isomorphism by~\cite{atiyah}*{Theorem 7}. Hence, it is an isomorphism over $k$. To see
    that $J_Y(1,d)$ is fine, note that $J_Y(r,d)$ is fine because of the existence of the
    kernel of the equivalence $\Drm^b(X)\we\Drm^b(Y)$. Taking the determinant of the kernel
    one obtains the universal object on $J_Y(1,d)\times_k Y$.
\end{proof}

\begin{lemma}
    If $X$ and $Y$ are genus $1$ curves over a field $k$, and if $\Drm^b(X)\simeq\Drm^b(Y)$,
    then $X$ and $Y$ are homogeneous spaces under the same elliptic curve. That is,
    $Jac(X)\iso Jac(Y)$.
\end{lemma}

\begin{proof}
    Realizing $X$ as $J_Y(1,d)$ for some $d$ by the previous lemma, we can define an action
    of $Jac(Y)=J_Y(1,0)$
    \[
    J_Y(1,0)\times J_Y(1,d) \to J_Y(1,d)
    \]
    by $(L, V)\mapsto L\otimes V$.
    This action is defined over $k$, but
    over $k^s$ we apply the corollary to \cite{atiyah}*{Theorem 7} to see that the kernel
    of the action is trivial. Moreover, the action is geometrically transitive by the same corollary.
    It follows that $Jac(Y)$ acts simply transitively on $X$ over $k$.
    But, this implies that $Jac(X)\iso Jac(Y)$ by~\cite{silverman}*{Theorem X.3.8}.
\end{proof}

By the lemma, we can focus our attention on derived equivalences among
principal homogeneous spaces for a fixed elliptic curve $E$.

\begin{theorem}\label{mainappendix}
    If $X$ and $Y$ are principal homogeneous spaces for an elliptic curve $E$ over a field $k$,
    then $\Drm^b(X)\simeq\Drm^b(Y)$ if and only if there exists an automorphism
    $\phi\in\Aut_k(E)$ and an integer $d$ of order coprime to $\mathrm{ord}([Y])$ in
    $\Hoh^1(k,E)$ such that $X=\phi_* dY$.
\end{theorem}

We will prove the theorem after two lemmas.

\begin{lemma}\label{lem:orbitlemma}
    If $X,Y\in\Hoh^1(k, E)$ are genus $1$ curves with principal homogeneous space structures
    under $E$, then $X$ and $Y$ are isomorphic as $k$-curves
    if and only if they are in the same orbit under the natural action of $\Aut_k(E)$ on
    $\Hoh^1(k,E)$.
\end{lemma}

Note that for any elliptic curve over any field the group $\Aut_k(E)$ is finite and of order
dividing $24$.

\begin{proof}
    Suppose that $\mu_0:E\times_kX\rightarrow X$ and $\mu_1:E\times_kY\rightarrow Y$ are
    the homogeneous space structures for $X$ and $Y$. The
    action of $\phi$ on $\Hoh^1(k,E)$ sends $(Y,\mu_1)$ to
    $(Y,\mu_1\circ(\phi\times\id_Y))$. The homogeneous spaces $(X,\mu_0)$ and
    $(Y,\mu_1\circ(\phi\times\id_X))$ are equal in the Weil-Ch\^atelet group $\Hoh^1(k,E)$ if
    and only if there is a $k$-isomorphism $f:X\rightarrow Y$ such that
    $f\circ\mu_0=\mu_1\circ(\phi\times\id_X)\circ(\id_E\times f)=\mu_1\circ(\phi\times f)$.
    In particular, this says that any two principal homogeneous spaces in the same
    $\Aut_k(E)$-orbit have isomorphic underlying curves. Now, given a $k$-isomorphism
    $f:X\rightarrow Y$, define $\phi:E\rightarrow E$ by
    $\mu_0(p,x)=\mu_1(\phi(p),f(x))$. Following the reasoning in~\cite{silverman}*{Section
    X.3}, one shows that $\phi$ is a $k$-automorphism of $E$. Hence, $X$ and $Y$
    are in the same $\Aut_k(E)$-orbit.
\end{proof}

\begin{lemma}\label{lem:multiplication}
    If $Y\in\Hoh^1(k,E)$ is a principal homogeneous space for $E$, then $dY$ is the homogeneous space
    $J_Y(1,d)$ of degree $d$ line bundles on $Y$.
\end{lemma}

\begin{proof}
    Suppose $E$ is an elliptic curve over $k$, and let $G_k$ be the Galois group
    of the separable closure $k^s$ over $k$. Let $\alpha\in\mathrm{Z}^1(G_k,E(k^s))$ be a
    cocycle representing $Y$. We assume that $\alpha$ is constructed as follows. Pick a
    point $p\in E(k^s)$. Then, $\alpha(g)=g\cdot p\ominus p$, where subtraction is via the
    $k$-linear map $Y\times_k Y\rightarrow E$. We adopt here a convention by which we use
    $\oplus$ and $\ominus$ for the group operations on $E$ or the homogeneous space
    structure on $Y$, reserving $+$ and $-$ for divisors. Hence, if $q$ and $r$ are points of $Y$,
    then $q\ominus r$ is the unique point $x$ of $E$ such that $x\oplus r=q$. The point $p$ defines
    an isomorphism $v_p:Y_{k^s}\rightarrow E_{k^s}$ via $q\mapsto q\ominus p$. One way to view the
    cocycle $\alpha$ is that it describes the difference between the $G_k$-actions on
    $E(k_s)$ using the natural action and using the isomorphism $v_p$. Indeed, given $x\in
    E(k^s)$, we can compute
    \begin{align*}
        v_p(g\cdot v_p^{-1}(x))&=v_p(g\cdot (x\oplus p))=v_p(g\cdot x\oplus g\cdot p)=g\cdot x\oplus g\cdot p\ominus p\\
        &=g\cdot x\oplus\alpha(g).
    \end{align*}
    Now, perform the same calculation for $J_Y(1,d)$. Namely, $v_p$ defines an isomorphism
    $J_{Y_{k^s}}(1,d)\rightarrow J_{E_{k^s}}(1,d)$, and we can compose with an isomorphism
    $J_{E_{k^s}}(1,d)\rightarrow E_{k^s}$ to find a cocycle representation of $J_Y(1,d)$.
    Let $e$ denote the identity element of $E(k^s)$. Then, we view $J_E(1,d)$ as being
    isomorphic to $J_E(1,0)\iso E$ via subtraction of the divisor $de$. By abuse of notation, we write $v_p$ for the
    composite isomorphism $J_{Y_{k^s}}(1,d)\rightarrow J_{E_{k^s}}(1,0)$. We compute for an
    element $x\in E(k^s)$ that
    \begin{align*}
        v_p(g\cdot v_p^{-1}(x-e))&=v_p(g\cdot ((x\oplus p)+(d-1)p))\\
        &=v_p((g\cdot x\oplus g\cdot p)+(d-1)(g\cdot p))\\
        &=(g\cdot x\oplus g\cdot p\ominus p)+(d-1)(g\cdot p\ominus p)-de\\
        &=g\cdot x+(g\cdot p\ominus p)-e+(g\cdot p\ominus p)^{\oplus (d-1)}+(d-2)e-de\\
        &=g\cdot x+(g\cdot p\ominus p)^{\oplus d}-2e\\
        &=g\cdot x\oplus (g\cdot p\ominus p)^{\oplus d}-e,
    \end{align*}
    using the definition of the group law on $E(k^s)$ and the fact that $e$ is $G_k$-fixed. The lemma follows.
% 
% %     Explicitly via Galois
% %     descent, we may hence describe $Y$ as the variety $Y_{k^s} = E_{k^s}$
% %     together with the twisted Galois action on the $k^s$ points given by
% %     $g_\alpha(x) = g(x) + \alpha(g)$ for $g \in
% %     G_k$.  Equivalently, we can describe its functor of points as follows:
% %     If $R$ is a $k$-algebra, and $R_{k^s}/R$ is the $G_k$-Galois extension
% %     of commutative rings given by scalar extension of $k^s/k$, we have
% %     \[Y(R) = \{x \in E(R_{k^s}) \ | \ \text{$g(x) + \alpha(g)
% %     = x$ for all $g\in G_k$}\},\]
% %     the fixed points for the twisted Galois action on $Y(R_{k^s})$.
% % 
% %     If $Y$ is a genus $1$ curve and $d$ is an integer,
% %     denote by $\ShPic^d(Y)$ the moduli stack of degree
% %     $d$ line bundles on $Y$ and by $J_Y(1,d)$ the corresponding coarse
% %     moduli space.
% 
%     The points of $J_Y(1,d)(k^s)$ are divisor classes, which we may represent by formal
%     sums $\sum n_ix_i$ where $\sum n_i=d$ and the $x_i\in Y(k^s)\iso Y(k^s)$. Now,
%     \begin{align*} \displaystyle
%         g_{\alpha}\left(\sum n_i x_i\right) &= \sum n_i(g(x_i)\oplus \alpha(g))\\
%         &= \bigoplus_i (\bigoplus^{n_i} g(x_i)) \oplus d \alpha(g)\\
%         &= g_{d\alpha}(\bigoplus_i (\bigoplus^{n_i} x_i)) \\
%         &= g_{d\alpha}(\phi(\sum n_i x_i)).
%     \end{align*}
%     This proves the lemma.
\end{proof}

Now, we prove the main result of the section.

\begin{proof}[Proof of Theorem~\ref{mainappendix}]
    Assume that $\Drm^b(X)\we\Drm^b(Y)$. Then, by Lemma~\ref{lem:jy1d}, as a variety, $X$ is
    isomorphic to $J_Y(1,d)$ for some $d$, but this isomorphism might not take into
    account the homogeneous space structures on each side. It follows from
    Lemma~\ref{lem:orbitlemma} that $X=\phi_*J_Y(1,d)$ for some $\phi\in\Aut_k(E)$, and now
    from Lemma~\ref{lem:multiplication} that $X=\phi_*dY$. As this argument is entirely
    symmetric, $Y=\psi_*eX$ for some $\psi\in\Aut_k(E)$. Hence, the orders of $Y$ and
    $\psi_*\phi_*edY$ in $\Hoh^1(k,E)$ agree. As the order of $\psi_*\phi_*edY$ is that of
    $edY$, we must have that $ed$ is coprime to $\mathrm{ord}([Y])$, as desired.

    Now, suppose that $X=\phi_*dY$ for some $\phi\in\Aut_k(E)$ and some $d$ prime to
    $\mathrm{ord}([Y])$. In order to prove that $\Drm^b(X)\we\Drm^b(Y)$, using
    Lemma~\ref{lem:orbitlemma}, we see that it is enough to consider the case when $\phi$ is
    the identity, which is to say, by Lemma~\ref{lem:multiplication}, when $X=J_Y(1,d)$.
    We conclude by showing that there is a universal sheaf on $J_Y(1,d)\times_k Y$, which induces
    an equivalence $\Drm^b(X)\we\Drm^b(Y)$.

    As $J_Y(1,d)$ is the moduli space of a $\Gm$-gerbe, the obstruction class is an element
    $\alpha\in\Br(J_Y(1,d))\subseteq\Br(K)$, where $K$ is the function field of $J_Y(1,d)$.
    In particular, $\alpha=0$ if and only if $\alpha_\eta=0$, where $\eta:\Spec K\rightarrow
    J_Y(1,d)$ is the generic point. If $d$ is relatively prime to $\mathrm{ord}([Y])$, then
    $Y$ and $J_Y(1,d)$ generate the same subgroup of $\Hoh^1(k,E)$. As $J_Y(1,d)$ is in the
    kernel of $\Hoh^1(k,E)\rightarrow\Hoh^1(K,E_K)$, it follows that $Y$ has a $K$-rational
    point as well. But, as soon as there is a $K$-rational point, $J_{Y_K}(1,d)$ is a fine
    moduli space. Using the rational point, there is a factorization of $\Br(J_Y(1,d))\rightarrow\Br(K)$
    through $\Br(J_Y(1,d))\rightarrow\Br(J_{Y_K}(1,d))$. Hence, $\alpha=0$.
\end{proof}

\subsection{Examples and Applications}

In this section we show that in some typical situations it is still the case that genus $1$ curves are derived equivalent if and only if they are isomorphic.
We also show that when this is not the case, it provides counterexamples to an open problem about moduli of stable vector bundles on genus $1$ curves.

\begin{example}
If $X$ and $Y$ are derived equivalent genus $1$ curves over a finite field $k=\FF_q$,
then $X\iso Y$.
\end{example}

\begin{proof}
Lang's theorem tells us in general that if $X/k$ is geometrically an abelian variety, then
$X(k)\neq \emptyset$. Thus $X$ and $Y$ are elliptic curves over $k$. Now by Theorem
\ref{thm:elliptic} we have that $X\iso Y$. 
\end{proof}

\begin{example}
If $X$ and $Y$ are derived equivalent genus $1$ curves over $\RR$, then $X\iso Y$.
\end{example}

\begin{proof}
Let $E=Jac(X)$. If $X$ and $Y$ are not isomorphic, then they must represent two distinct
$\Aut_k(E)$-orbits in $\Hoh^1(\RR, E)$. Our goal is to show that $\Hoh^1(\RR, E)$ is
too small to allow non-isomorphic torsor classes. There are two cases corresponding to
whether or not $E$ has full $2$-torsion defined over $\RR$. We merely check that
$\big|\Hoh^1(\RR, E)\big|\leq 2$. This suffices because the two classes cannot be
derived equivalent by Theorem \ref{thm:elliptic}.

Since $\Gal(\CC/\RR)\iso\ZZ/2$,
we know that $\Hoh^1(\RR, E)$ is killed by $2$.
This tells us computing the whole group is equivalent to computing the $2$-torsion part.
In the case that $E$ has full $2$-torsion defined over $\RR$, consider the Kummer sequence
\[
0\to E[2] \to E \stackrel{2}{\to} E \to 0.
\]
This induces an exact sequence
\[
0\to E(\RR)/2E(\RR)\to\Hoh^1(\RR, E[2])\to\Hoh^1(\RR, E)\to 0
\]
The middle group consists of non-twisted homomorphisms $\Hom(\ZZ/2, \ZZ/2\times \ZZ/2)$ because the $2$-torsion is fully defined over $\RR$, so the Galois action is trivial. Since the Weil pairing is non-degenerate and Galois invariant, we know that the full $4$-torsion is not defined over $\RR$ otherwise $\RR^\times$ would contain four distinct roots of unity. This means that $[2]: E(\RR)\to E(\RR)$ is not surjective and hence $E(\RR)/2E(\RR)$ is non-trivial. This proves the inequality.

The other case is that $E[2](\RR)\iso \ZZ/2$.
In this case we can explicitly write down in terms of elements $E[2]=\{1,a,b,c\}$.
Without loss of generality, we assume $a^\sigma=a$, $b^\sigma=c$ and $c^\sigma=b$ where $\sigma$ is the non-trivial element of $G_\RR$.
The condition on $\rho:G_\RR\to E[2]$ being a twisted homomorphism forces $\rho(1)=1$ and $\rho(\sigma)=1$ or $a$.
Thus there are only two possible cocycles. The non-trivial one is actually also a coboundary, since $\rho(\sigma)=b^\sigma-b$.
Thus $H^1(\RR, E[2])=0$, which forces $H^1(\RR, E)=0$.
\end{proof}

\begin{example}
There exist non-isomorphic derived equivalent genus $1$ curves.
\end{example}

\begin{proof}
Fix $E/\QQ$, a non-CM elliptic curve with $j(E)\neq 0, 1728$.
Consider a genus $1$ curve $X\in\Hoh^1(\QQ, E)$ with period $5$.
The cyclic subgroup generated by $X$ has order $5$ and hence all four non-split classes are generators.
Only one other of these generators can be isomorphic as a $\QQ$-curve by Lemma \ref{lem:orbitlemma}.
But by Theorem~\ref{mainappendix} any non-isomorphic generator is a non-isomorphic derived equivalent curve.
\end{proof}

\begin{example}\label{curvemodspace}
For any $N>0$, there exists a genus $1$ curve $Y$ that admits at least $N$ distinct moduli
spaces of stable vector bundles $J_Y(r,d)$.
\end{example}

\begin{proof}
We can again fix $E/\QQ$ a non-CM elliptic curve with $j(E)\neq 0, 1728$ and $N>0$.
Choose a prime $p>3N$.
There exists a cyclic subgroup of $\Hoh^1(\QQ, E)$ of order $p$.
Since $\Aut_\QQ(E)\iso \ZZ/2$, there are more than $p/2>N$ non-isomorphic generators.
By Theorem~\ref{mainappendix}, any two of these generators,
$X$ and $Y$, are derived equivalent.
Thus $X\iso J_Y(r,d)$ for some $r$ and $d$ coprime.
\end{proof}

This result is a rather surprising contrast to the results of Atiyah which says that fine moduli spaces of vector bundles on an elliptic curve are always isomorphic to the elliptic curve.
More recently, Pumpl\"{u}n~\cite{pumplun}
even extended some of these results to work in a more general genus $1$ setting suggesting that these
types of examples might not exist.

\section{Background on dg categories}\label{sec:dg}

In this section we give a brief introduction to dg categories. For details and further
references, consult Keller~\cite{keller-icm}.

\subsection{Definitions}

A \emph{dg category} $\Cscr$ over a commutative ring $R$, also called an $R$-linear dg
category, consists of
\begin{enumerate}
    \item   a class of objects $\mathrm{ob}(\Cscr)$;
    \item   for each pair of objects $x$ and $y$ a chain complex $\Map_\Cscr(x,y)$ of
        $R$-modules;
    \item   for each triple of objects $(x,y,z)$ a morphism of degree $0$
        $$\Map_\Cscr(y,z)\otimes_R\Map_\Cscr(x,y)\rightarrow\Map_\Cscr(x,z)$$
        of chain complexes of $R$-modules
\end{enumerate}
satisfying obvious analogues of the unit and associativity axioms of a category. Note that
the tensor product is the underived tensor product. For the sake of concreteness, we fix the
sign conventions appearing in Keller's ICM talk~\cite{keller-icm}.

When $\mathrm{ob}(\Cscr)$ is a set (and not a proper class), we say that $\Cscr$ is small.
An example is the dg
category $\Cscr$ consisting of a single point $\ast$ where $\Map_\Cscr(\ast,\ast)$ is
any fixed dg algebra.  Let $\mathrm{dgcat}_R$
denote the category of small dg categories over $R$. A morphism in $\mathrm{dgcat}_R$ is a
dg functor, i.e., the data of a function $F:\mathrm{ob}(\Cscr)\rightarrow\mathrm{ob}(\Dscr)$
together with functorial morphisms of chain complexes
$\Map_\Cscr(x,y)\rightarrow\Map_\Dscr(F(x),F(y))$ for all objects $x,y$.

The homotopy category $\Ho(\Cscr)$ of a dg category $\Cscr$ is the additive $R$-linear
category obtained by taking
the same objects as $\Cscr$, but where $\Hom_{\Ho(\Cscr)}(x,y)=\Hoh^0\Map_{\Cscr}(x,y)$.
A quasi-equivalence of dg categories is a dg functor $F:\Cscr\rightarrow\Dscr$ such that
$\Map_{\Cscr}(x,y)\rightarrow\Map_{\Dscr}(F(x),F(y))$ is a quasi-equivalence for all
$x,y\in\Cscr$ and such that $\Ho(F):\Ho(\Cscr)\rightarrow\Ho(\Dscr)$ is essentially
surjective. Note that these conditions imply that $\Ho(F)$ is an equivalence of categories.

Another construction is $\mathrm{Z}^0(\Cscr)$, a category with the same class of objects as
$\Cscr$, but in which $\Hom_{\mathrm{Z}^0(\Cscr)}(x,y)=\mathrm{Z}^0\Map_\Cscr(x,y)$.

% If $\Ascr$ is an abelian category, then
% $\mathrm{Ch}_{\dg}(\Ascr)$ is the dg category of chain complexes of objects in $\Ascr$, and
% $\Ho(\Ch_{\dg}(\Ascr))$ is equivalent to the category of chain complexes in $\Ascr$ up to homotopy.

% The derived category $\Drm(\Cscr)$ is the localization $\Ho(\Cscr)[S^{-1}]$ where $S$ is the
% class of morphisms $f:x\rightarrow y$ in $\Hoh^0\Map_{\Cscr}(x,y)$ such that
% $f_*:\Map_{\Cscr}(z,x)\rightarrow\Map_{\Cscr}(z,y)$ is a quasi-isomorphism for all
% $z\in\Cscr$, assuming that this localization exists.

% When $\Ascr$ is a Grothendieck abelian category, it is possible to show that
% $\Drm(\Ch_{\dg}(\Ascr))$
% exists using model category theoretic methods; this was apparently done first by Joyal in
% his famous letter to Grothendieck. There is a
% model category structure on $\Ch_\dg(\Ascr)$ in which every object is cofibrant and
% where the fibrant objects are the bounded-above complexes of injectives. Then,
% $\Drm(\mathrm{Ch}_\dg(\Ascr))$
% can be defined as $\Ho(\Ch_\dg(\Ascr)^{\circ})$, the homotopy category of the dg
% category of fibrant and cofibrant complexes of objects of $\Ascr$. This is spelled out in
% the language of the differential graded nerve in~\cite{ha}*{Section 1.3.5}.
% Note that the natural map $\Ch_\dg(\Ascr)^\circ\rightarrow\Ch_\dg(\Ascr)$ is a
% quasi-equivalence of dg categories.

\subsection{Big and small dg categories}

The key notion in the study of dg categories is the idea of a
right module over a small dg category $\Cscr$. This is simply a dg functor
\begin{equation*}
    M:\Cscr^{\op}\rightarrow\Ch_{\dg}(R),
\end{equation*}
where $\Ch_{\dg}(R)$ is the dg category of complexes of $R$-modules.
The right modules over $\Cscr$ form a dg category we will call $\Mod_\dg(\Cscr)$. We refer
to~\cite{keller-icm} for the definition of the mapping complexes in $\Mod_\dg(\Cscr)$, and
for the important projective model category structure on $\mathrm{Z}^0\Mod_\dg(\Cscr)$.

In the special case of a dg algebra $A$ viewed as a dg category with one object, giving a
right module $M:A^{\op}\rightarrow\Ch_\dg(R)$ is the same as giving a chain complex of $R$-modules
$M$ together with a map of dg algebras $A^{\op}\rightarrow\End_R(M)$. That is, we recover
the usual sense of right $A$-module.

There is a model category structure on $\mathrm{dgcat}_R$ in which the weak equivalences
are the quasi-equivalences of dg categories~\cite{tabuada-une-structure}. For the purposes of this paper, the derived
category of a small $R$-linear dg category $\Cscr$ is the dg category
$\Drm_\dg(\Cscr)=\Mod_\dg(\mathbf{c}\Cscr)^\circ$, the full dg subcategory of cofibrant
objects (with respect to the projective model structure) in $\Mod_\dg(\mathbf{c}\Cscr)$, where $\mathbf{c}\Cscr$ is a cofibrant replacement
for $\Cscr$ in $\mathrm{dgcat}_R$. This is a large dg category.

The dg categories $\Drm_\dg(\Cscr)$ that arise in this way are very special. Their homotopy
categories $\Drm(\Cscr)=\Ho(\Drm_\dg(\Cscr))$ are triangulated categories that
are closed under arbitrary coproducts,
compactly generated, and locally small. We call any such dg category a \emph{compactly generated
stable presentable} $R$-linear dg category. This
terminology differs slightly from that used in most literature on dg categories, and is
derived instead from the language of stable $\infty$-categories, as developed in~\cite{ha}.

Here is another way to describe $\Drm(\Cscr)$.
A map $M\rightarrow N$ of right modules over $\Cscr$ is a quasi-isomorphism if the induced
map $M(X)\rightarrow N(X)$ is a quasi-isomorphism in $\Ch_{\dg}(R)$ for each object $X$ of
$\Cscr$. The derived category of $\Cscr$ is the localization of $\Ho(\Mod_\dg(\Cscr))$ at
the quasi-isomorphisms. It is in fact equivalent to $\Drm(\Cscr)$. Besides providing a dg
model for $\Drm(\Cscr)$, the construction of $\Drm_\dg(\Cscr)$ above serves to show that
this localization actually exists. Details can be found in~\cite{keller-icm}.

Note the correspondence between the small dg category $\Cscr$ and its category of right
modules $\Mod_\dg(\Cscr)$. This correspondence becomes tighter if we use pretriangulated small
dg categories. Call a small dg category $\Cscr$ pretriangulated if the
image of the Yoneda embedding $\Ho(\Cscr)\rightarrow\Drm(\Cscr)$ is stable under shifts
and extensions. A pretriangulated small dg category $\Cscr$ is idempotent-complete if this
image is also stable under summands.

\subsection{Sheaves and dg categories}

If $X$ is an $R$-scheme,
there is an $R$-linear dg category $\QC_{\dg}(X)$ of complexes of $\Oscr_X$-modules with
quasi-coherent cohomology sheaves. This dg category has a full dg subcategory
$\Perf_{\dg}(X)\subseteq\QC_{\dg}(X)$ consisting of the perfect complexes. These are the
complexes of $\Oscr_X$-modules that are Zariski-locally quasi-isomorphic to bounded
complexes of vector bundles. By a theorem of Bondal and van den
Bergh~\cite{bondal-vandenbergh}, $\QC_{\dg}(X)$ is
quasi-equivalent to $\Drm_\dg(\Perf_\dg(X))$.

There is an equivalence $\Drm(\QC_{\dg}(X))\we\Drm_{\qc}(X)$, where $\Drm_{\qc}(X)$ is the
usual triangulated category of complexes of $\Oscr_X$-modules with quasi-coherent cohomology
sheaves. The full subcategory of $\Drm(\QC_{\dg}(X))$ consisting of objects quasi-isomorphic
to objects in $\Perf_\dg(X)$ is $\Perf(X)$, the triangulated category of perfect complexes. If $X$ is quasi-compact and
quasi-separated, the perfect complexes $\Perf(X)\subseteq\Drm_{\qc}(X)$ have a purely
categorical description. Namely, by Bondal and van den Bergh~\cite{bondal-vandenbergh}*{Theorem 3.1.1}
they are the complexes $x\in\Drm_{\qc}(X)$ such that the
functor $\Hom_{\Drm_{\qc}(X)}(x,-):\Drm_{\qc}(X)\rightarrow\Mod_R$ commutes with coproducts.

When $X$ is regular and noetherian, the natural inclusion of triangulated categories
$\Perf(X)\rightarrow\Drm^b(X)$ is an equivalence. This is because any bounded complex of
coherent $\Oscr_X$-modules has a finite-length resolution by vector bundles Zariski-locally.

\begin{proposition}\label{prop:triadg}
    If $X$ and $Y$ are quasi-compact and quasi-separated and if
    $\Drm_{\qc}(X)\we\Drm_{\qc}(Y)$, then $\Perf(X)\we\Perf(Y)$. If $X$ and $Y$ are
    smooth projective schemes over a field $k$, then the following are equivalent:
    \begin{enumerate}
        \item   $\Drm_{\qc}(X)\we\Drm_{\qc}(Y)$ as $k$-linear triangulated categories;
        \item   $\Drm^b(X)\we\Drm^b(Y)$ as $k$-linear triangulated categories;
        \item   $\Perf_\dg(X)\we\Perf_\dg(Y)$ as $k$-linear dg categories;
        \item   $\QC_\dg(X)\we\QC_\dg(Y)$ as $k$-linear dg categories.
    \end{enumerate}
    \begin{proof}
        If $\Drm_{\qc}(X)\we\Drm_{\qc}(Y)$, then the subcategories of compact
        objects coincide, as these are preserved by any equivalence of triangulated
        categories. Since these subcategories are precisely the categories of perfect
        complexes, we have $\Perf(X)\we\Perf(Y)$. If $X$ and $Y$ are regular and noetherian,
        then $\Drm_{\qc}(X)\we\Drm_{\qc}(Y)$ implies $\Drm^b(X)\we\Drm^b(Y)$ by the first
        part and by our remark in the previous paragraph. Hence, (1) implies (2).

        Suppose now that $F:\Drm^b(X)\rightarrow\Drm^b(Y)$ is an equivalence. Then, by an
        important theorem of Orlov~\cite{orlov-equivalences},
        there is a complex $P\in\Drm^b(X\times_k Y)$ such that the Fourier-Mukai functor
        $\Phi_P:\Drm^b(X)\rightarrow\Drm^b(Y)$ agrees with $F$, where $\Phi_P$ is defined by
        \begin{equation*}
            \Phi_P(x)=\mathrm{R}\pi_{Y,*}(P\otimes^{\Lrm}\pi_X^*x),
        \end{equation*}
        and where $\pi_X$ and $\pi_Y$ denote the projections from $X\times_k Y$.
        Because we have this nice model for $F$, $\Phi_P$ extends to a functor
        $\Perf_{\dg}(X)\rightarrow\Perf_{\dg}(Y)$ which is by definition a
        quasi-equivalence, so that (2) implies (3). That (3) implies (4) is immediate from
        the description of $\QC_\dg(X)$ as $\Drm_\dg(\Perf_\dg(X))$. Finally, that (4)
        implies (1) follows by taking homotopy categories.
    \end{proof}
\end{proposition}

\begin{remark}
    Because of the second statement of the proposition, we work everywhere with
    $\Drm_{\qc}(X)$ and its dg enhancement $\QC_\dg(X)$ in the rest of this paper.
\end{remark}

\begin{remark}
    The converse to the first statement would hold if we required a possibly stronger
    condition, namely that the dg enhancements $\Perf_{\dg}(X)$ and $\Perf_{\dg}(Y)$ are
    quasi-equivalent. Indeed, $\QC_{\dg}(X)$ can be constructed from $\Perf_{\dg}(X)$ as we
    have remarked above.
\end{remark}

\subsection{Base change}

If $\Cscr$ is a small $R$-linear dg category and $S$ is a commutative $R$-algebra, we denote
by $\Cscr_S$ the $S$-linear dg category with the same objects as $\Cscr$, but where
\begin{equation*}
    \Map_{\Cscr_S}(x,y)=\Map_\Cscr(x,y)\otimes_R S
\end{equation*}
for $x,y\in\mathrm{ob}(\Cscr)$.

If $\Qrm_\dg=\Drm_\dg(\Cscr)$, we define $\Qrm_{\dg,S}=\Drm_\dg(\Cscr_S)$. If $\Cscr$ is a
small idempotent-complete pretriangulated dg category, we define $\Cscr_S$ as the full
subcategory of compact objects in $\Qrm_{\dg,S}$. Note that in this case we have now
overloaded the definition of the base change of a small idempotent-complete pretriangulated
dg category. We will however always mean the latter when $\Cscr$ is of this form.

We choose this definition so that the
result is also pretriangulated and idempotent complete, and so that the following examples
hold. If $A$ is an $R$-algebra, then $\Drm_{\dg}(A)_S\we\Drm_\dg(A\otimes_R^{\Lrm}S)$.
If either $X$ is a flat $R$-scheme or $S$ is a flat commutative $R$-algebra, then $\QC_{\dg}(X)_S\we\QC_\dg(X_S)$, and
$\Perf_\dg(X)_S=\Perf_\dg(X_S)$.

\subsection{Stacks of dg categories}\label{sec:stacks}

Later in the paper, we will need to use stacks of dg categories. We take the perspective
of~\cite{toen-derived}*{Section 3}, and explain briefly what we mean here.

A stack of stable presentable dg categories $\StQ_\dg$ on a scheme $X$ in some topology $\tau$ is an assignment of
a stable presentable dg category (what To\"en calls locally presentable)
\begin{equation*}
    \StQ_\dg(\Spec S)
\end{equation*}
for each map $\Spec S\rightarrow X$,
together with the assignment of pullback maps $f^*:\StQ_\dg(\Spec S)\rightarrow\StQ_\dg(\Spec
T)$ that preserve homotopy colimits
for each map $f:\Spec T\rightarrow\Spec S$ of affine $X$-schemes. One needs to fix moreover
the various data that encode the composition functions and so on. Finally, one requires
that whenever $S\rightarrow T^\bullet$ is a hypercover in the $\tau$-topology, the natural
map
\begin{equation*}
    \StQ_\dg(\Spec S)\rightarrow\mathrm{holim}_\Delta\StQ_\dg(\Spec T^\bullet)
\end{equation*}
is a quasi-equivalence. This homotopy limit is taken in an appropriate model category or
$\infty$-category of big $R$-linear dg categories. See To\"en~\cite{toen-derived}*{Section
3} for details and references.

There are a few remarks about how we will use these that need to be made. First of all,
define a stable presentable $R$-linear dg category with descent to be a stable presentable
$R$-linear category $\Qrm_\dg$ such that if $R\rightarrow S$ is a map of commutative rings,
and if $S\rightarrow T^\bullet$ is a $\tau$-hypercover, then the natural map
\begin{equation*}
    \Qrm_{\dg,S}\rightarrow\mathrm{holim}_\Delta\Qrm_{\dg,T^\bullet}
\end{equation*}
is a quasi-equivalence. If
$X=\Spec R$ is itself affine, then giving a stack of stable presentable dg categories
$\StQ_\dg$ is equivalent to giving the stable presentable $R$-linear dg category with
descent $\Qrm_\dg=\Qrm_\dg(\Spec R)$.

Second, any dg category $\Drm_{\dg}(\Cscr)$, where $\Cscr$ is a small $R$-linear dg category,
is a stable presentable $R$-linear dg category with fpqc descent. This important fact follows
from~\cite{toen-derived}*{Corollary 3.8}.

Third, as a consequence, if $X$ is a quasi-compact and quasi-separated $R$-scheme, the dg category $\QC_\dg(X)$ 
is a stable presentable $R$-linear dg category with descent. Indeed, in this case,
$\QC_\dg(X)=\Drm_{\dg}(\Perf_{\dg}(X))$ by~\cite{bondal-vandenbergh}. It follows that if
$X\rightarrow B$ is a flat quasi-compact and quasi-separated morphism of schemes,
then the functor $\StQC^X$ on $B$, which assigns to
any $\Spec S\rightarrow B$ the $S$-linear dg category $\StQC^X_\dg(\Spec S)=\QC_{\dg}(X_S)$
is an fpqc stack of stable presentable dg categories on $B$.

Fourth, by restricting the class of affines we test on, given a stack $\StQ_\dg$ of stable presentable
$R$-linear categories on $X$ and a map $f:Y\rightarrow X$, we can obtain a stack
$f^*\StQ_\dg$ on $X$.

Fifth, there is an obvious notion of an equivalence of stacks $\StQ_\dg\we\mathscr{P}_\dg$,
which results by specifying equivalences $\StQ_\dg(\Spec S)\rwe\mathscr{P}_\dg(\Spec S)$,
specifying compatibilities with the pullback functors in each stack, and specifying various
higher homotopy coherences. Of particular importance is that if $X$ and $Y$ are $R$-schemes
at least one of which is flat over $R$, then any complex $P$ in $\Drm_{\qc}(X\times_R Y)$ gives rise
to a Fourier-Mukai functor $\Phi_P:\StQC^X\rightarrow\StQC^Y$ by the usual formula.

\begin{proposition}\label{prop:localequivalence}
    Let $\Phi:\QC_{\dg}(X)\rightarrow\QC_{\dg}(Y)$ be a morphism of dg categories over $k$ such that
    $\Phi_{k^s}:\QC_{\dg}(X_{k^s})\rightarrow\QC_{\dg}(Y_{k^s})$ is an equivalence. Then, $\Phi$ is an equivalence.
    \begin{proof}
        This follows from the fact that the assignment $l\mapsto\QC_{\dg}(X_l)$ is
        actually an fpqc stack of stable presentable dg categories on $\Spec k$ by
        To\"en~\cite{toen-derived}*{Proposition 3.7}. In particular, there is a commutative diagram
        \begin{equation*}
            \xymatrix{
            \QC_{\dg}(X)\ar[r]\ar[d]^{\Phi}    &
            \mathrm{holim}_{\Delta}\QC_{\dg}(X_{(k^s)^{\otimes
            n}})\ar[d]^{\Phi_{(k^s)^{\otimes n}}}\\
            \QC_{\dg}(Y)\ar[r]                 &
            \mathrm{holim}_{\Delta}\QC_{\dg}(Y_{(k^s)^{\otimes n}})
            }
        \end{equation*}
        in which the horizontal arrows are equivalences, and where the tensor products
        $(k^s)^{\otimes n}$ are taken over $k$. By base change, the right-hand vertical
        arrow is also an equivalence by our hypothesis. It follows that $\Phi$ is an equivalence.
    \end{proof}
\end{proposition}

\section{The twisted Brauer space}\label{sec:tbs}

This paper can be seen as a contribution to the arithmetic theory of derived categories.
To begin this study, one must understand how to twist a given derived category. This problem
was posed in~\cite{antieau-tbs}, where the twisted Brauer space was introduced as a
computational tool for classifying twists\footnote{In~\cite{antieau-tbs}, stable
$\infty$-categories were used instead of dg categories. No changes are needed to carry out
the same arguments in the dg setting.}. As motivation, consider the following fact due to
To\"en~\cite{toen-derived}: the Brauer
group $\Br(k)$ of a field classifies the stable presentable
$k$-linear categories with descent $\Qrm_\dg$ such that there is a finite separable field
extension $k\rightarrow l$ where $\Qrm_{\dg,l}\we\Drm_\dg(l)\we\QC_{\dg}(\Spec l)$.

More generally, over a quasi-compact and quasi-separated scheme $X$,
stacks of locally presentable dg categories on $X$ that are \'etale locally equivalent to
the canonical stack
$\QC_{\dg}$ are classified up to equivalence of stacks by the
derived Brauer group $\Hoh^1_{\et}(X,\ZZ)\times\Hoh^2_{\et}(X,\Gm)$. Note that we use the
entire group $\Hoh^2_{\et}(X,\Gm)$. If $X$ is normal, then $\Hoh^1_{\et}(X,\ZZ)=0$. If $X$
is regular and noetherian, then in addition
$\Hoh^2_{\et}(X,\Gm)_{\tors}=\Hoh^2_{\et}(X,\Gm)$, so that the derived Brauer group of $X$
coincides with the cohomological Brauer group in this case.

Let $Y\rightarrow X$ be a flat quasi-compact and quasi-separated morphism of schemes, and let $\StQC^Y_{\dg}$ be the associated stack of
stable presentable dg categories on $X$, as described in Section~\ref{sec:stacks}, which we
view as an \'etale stack on $X$.
Another \'etale stack $\StQ_{\dg}$ of locally presentable dg categories over $X$ is \'etale
locally equivalent to $\StQC_{\dg}^Y$
if there is an \'etale cover $p:\Spec S\rightarrow X$
such that $p^*\StQ_{\dg}\we p^*\StQC_{\dg}^Y$ as $S$-linear dg categories.
Equivalently, we should have $\StQ_\dg(\Spec S)\we\Drm_\dg(Y_S)$.

\begin{definition}
    Let $f:Y\rightarrow X$ be a flat map of schemes. Define $\Br^Y(X)$ to be the
    set of \'etale stacks of locally presentable dg categories on $X$ that are \'etale locally
    equivalent to $\StQC_{\dg}^Y$ modulo equivalence of stacks over $X$. We view $\Br^Y(X)$
    as a pointed set, with the point being the equivalence class defined by $\StQC_{\dg}^Y$.
\end{definition}

\begin{example}[\cite{antieau-tbs}]
    Suppose that $C$ is the genus $0$ curve defined by the equation $x^2+y^2+z^2=0$ 
    in $\PP^2_{\RR}$. Then, $C(\RR)=\emptyset$, and $C$ is not rational. In particular,
    $\Drm_{\qc}(C)$ is not equivalent to $\Drm_{\qc}(\PP^1_\RR)$. However, after extension to the
    complex numbers, we do have $\Drm_{\qc}(C_{\CC})\we\Drm_{\qc}(\PP^1_\CC)$. Hence,
    $\StQC_{\dg}^C$ defines a non-trivial element $\Br^{\PP^1}(\RR)$.
\end{example}

This pointed set was defined and realized as the set of connected components
$\pi_0\ShBr^Y(X)$ in~\cite{antieau-tbs}, where
$\ShBr^Y(X)$ is a certain topological space. In fact, this space is itself the space of global
sections of an \'etale hypersheaf of spaces on the big \'etale site of $X$.
This observation is useful for actually computing $\Br^Y(X)$.

As discussed in~\cite{antieau-tbs}*{Section 3.2},
there is a fibre sequence of sheaves of spaces on $X$
\begin{equation*}
    K(\THHB^0(Y)^\times,2)\rightarrow\ShBr^Y\rightarrow K(\ShAut_{\QC_{\dg}(Y)},1),
\end{equation*}
where $K(A,n)$ denotes the Eilenberg-MacLane sheaf of spaces for the sheaf $A$,
$\THHB^0(Y)^\times$ denotes the sheaf of units in the degree $0$ Hochschild cohomology ring
of $Y$, and $\ShAut_{\QC_{\dg}(Y)}$ is the sheaf of autoequivalences of the stack
$\StQC_{\dg}^Y$ on
$X$. Since $\THHB^0(Y)^\times$ is a sheaf of abelian groups, when the action of
$\ShAut_{\QC_{\dg}(Y)}$ on $\THHB^0(Y)^\times$ is trivial the sequence can be delooped,
and we can identify $\ShBr^Y$ with the fiber in
\begin{equation*}
    \ShBr^Y\rightarrow K(\ShAut_{\QC_{\dg}(Y)},1)\rightarrow K(\THHB^0(Y)^\times,3).
\end{equation*}

Recall that if $K(A,n)$ is an Eilenberg-MacLane sheaf, then
\begin{equation*}
    \pi_i \Gamma(X,K(A,n))\iso\Hoh^{n-i}(X,A)
\end{equation*}
for $0\leq i\leq n$ and $0$ otherwise. Suppose now that $Y$ is smooth, proper, and
geometrically connected over $X$. Then,
$\THHB^0(Y)^\times\iso\mathbb{G}_{m,X}$. By taking sections of the fiber sequence above, and
then taking the long exact sequence in homotopy, we obtain the exact sequence
\begin{gather*}
    0\rightarrow\mathbb{G}_{m,X}(X)\rightarrow\pi_2\ShBr^Y(X)\rightarrow 0\\
    \rightarrow\Hoh^1_{\et}(X,\mathbb{G}_{m,X})\rightarrow\pi_1\ShBr^Y(X)\rightarrow\ShAut_{\QC_{\dg}(Y)}(X)\rightarrow\Hoh^2_{\et}(X,\mathbb{G}_{m,X})\\
    \rightarrow\pi_0\ShBr^Y(X)\rightarrow\Hoh^1_{\et}(X,\ShAut_{\QC_{\dg}(Y)})\rightarrow\Hoh^3_{\et}(X,\mathbb{G}_{m,X}).
\end{gather*}
Exactness is a slightly touchy matter here, as the last line is an exact sequence of pointed
sets. It means that there is an action of $\Hoh^2_{\et}(X,\mathbb{G}_{m,X})$ on
$\pi_0\ShBr^Y(X)$ and the fibers of
$\pi_0\ShBr^Y(X)\rightarrow\Hoh^1_{\et}(X,\ShAut_{\QC_{\dg}(Y)})$ are precisely the orbits.
Exactness at $\Hoh^1_{\et}(X,\ShAut_{\QC_{\dg}(Y)})$ says only that the fiber of the map to
$\Hoh^3_{\et}(X,\mathbb{G}_{m,X})$ over $0$ is the image of $\pi_0\ShBr^Y(X)$.

The action of $\Hoh^2_{\et}(X,\Gm)$ on $\pi_0\ShBr^Y(X)$ can be described as follows. Given
$\alpha\in\Hoh^2_{\et}(X,\Gm)$ and $\StQ_\dg\in\pi_0\ShBr^Y(X)$, we simply form the tensor
product $\StQC^\alpha_\dg\otimes\StQ_\dg$, where $\StQC^\alpha_\dg$ is the stack that
assigns to $\Spec S\rightarrow X$ the dg category $\QC_{\dg}(\Spec S,\alpha)$, the dg category of
complexes of $\alpha$-twisted $\Oscr_X$-modules with quasi-coherent $\alpha$-twisted
cohomology sheaves. The basic example is
$\StQC_\dg^\alpha\otimes\StQC^X_\dg\we\StQC^{(X,\alpha)}$, the stack whose sections
over $\Spec S\rightarrow X$ is the dg category $\QC_{\dg}(X_S,\alpha)$. Since $\alpha$ and $\Cscr$ are both \'etale locally trivial, so is their
tensor product.

\section[Principal homogeneous spaces]{Derived equivalences of principal homogeneous spaces}\label{sec:derivedhomogeneous}

Orlov~\cite{orlov-abelian} and Polishchuk~\cite{polishchuk} have demonstrated that the group
$\Urm(A\times_k\hat{A})$ of isometric automorphisms of $A\times_k\hat{A}$ plays a central role
in the study of derived autoequivalences of $\Drm^b(A)$. Recall that $\Urm(A\times_k\hat{A})$
is the group of automorphisms
\begin{equation*}
    \sigma=\begin{pmatrix}
        x   &   y\\
        z   &   w
    \end{pmatrix}
\end{equation*}
of the abelian $k$-variety $A\times_k\hat{A}$ such that
\begin{equation*}
    \sigma^{-1}=\begin{pmatrix}
        \hat{w}   &   -\hat{y}\\
        -\hat{z}   &  \hat{x} 
    \end{pmatrix},
\end{equation*}
where $x$ is a homomorphism $A\rightarrow A$, $y$ is a homomorphism $\hat{A}\rightarrow A$, and so forth.

Orlov~\cite{orlov-abelian} showed that there is a representation of $\Aut_{\Drm^b(A)}$ on
$\Urm(A\times_k\hat{A})$ with kernel precisely the subgroup
$\ZZ\times (A\times_k\hat{A})(k)$. Moreover, when $k$ is algebraically closed, the map
$\Aut_{\Drm^b(A)}\rightarrow\Urm(A\times_k\hat{A})$ is surjective. This follows from
Orlov~\cite{orlov-abelian} when $k$ has characteristic $0$, and from
Polishchuk~\cite{polishchuk}*{Theorem 15.5} in general.

Note that $\Aut_{\Drm^b(A)}$ is isomorphic to $\Aut_{\QC_{\dg}(A)}$. Indeed, by Orlov's
representability theorem, every derived automorphism of $\Drm^b(A)$ has a kernel and hence
lifts to an automorphism of $\Perf_{\dg}(A)$. But, this then extends to an autoequivalence
of $\QC_{\dg}(A)$. The inverse map is given by observing that any autoequivalence of
$\QC_{\dg}(A)$ must preserve compact objects.

Returning to the case where $k$ is an arbitrary field,
from the work of Orlov, we have an exact sequence
\begin{equation*}
    0\rightarrow\ZZ\times A\times_k\hat{A}\rightarrow\ShAut_{\QC_{\dg}(A)}\rightarrow\Ubf^A
\end{equation*}
of sheaves of groups over $\Spec k$, where $\Ubf^A$ denotes the sheaf with
$\Ubf^A(l)=\Urm((A\times_k\hat{A})_l)$ for an extension $l/k$. Since the right-hand map is
surjective on algebraically closed fields, it follows that when $k$ is perfect, the map is
a surjective map of sheaves. In any case, let $\Vbf^A$ denote the image as an \'etale sheaf. So,
$\Vbf^A$ is a subsheaf of $\Ubf^A$, and its sections over a field $l$ is some
subgroup of the group of unitary automorphisms of $(A\times_k\hat{A})_l$.
We have an exact sequence
\begin{equation}\label{eq:gammaext}
    0\rightarrow\ZZ\times
    A\times_k\hat{A}\rightarrow\ShAut_{\QC_{\dg}(A)}\rightarrow\Vbf^A\rightarrow 1
\end{equation}
of \'etale sheaves of groups on $\Spec k$.

Recall from~\cite{serre-cohomologie}*{Section I.5.5} that in this setting there is an action of
$\Vbf^A(l)\subseteq\Urm((A\times_k\hat{A})_l)$ on $\Hoh^1_{\et}(l,A\times_k\hat{A})$, which
we will denote by $\sigma\diamond\begin{pmatrix}X\\Y\end{pmatrix}$, when $\sigma$ is a
unitary isomorphism defined over $l$, $X$ is a principal homogeneous space for $A_l$, and
$Y$ is a principal homogeneous space for $\hat{A}_l$. Note that
$\Urm((A\times_k\hat{A})_l)$ also acts in an obvious way on
$\Hoh^1_{\et}(l,A\times_k\hat{A})$, via automorphisms of the coefficients. We will denote
this action by $\sigma\cdot\begin{pmatrix}X\\Y\end{pmatrix}$. These two actions coincide if
and only if the boundary map
\begin{equation*}
    \delta:\Vbf^A(l)\rightarrow\Hoh^1_{\et}(l,A\times_k\hat{A})
\end{equation*}
vanishes by~\cite{serre-cohomologie}*{Proposition 40}.

The following theorem and its corollaries give the main application of the theory of twisted Brauer spaces in our
paper.

\begin{theorem}\label{thm:twisted}
    Let $A$ be an abelian variety over a field $k$. Suppose that $X$
    and $Y$ are principal homogeneous spaces for $A$. If $\Drm^b(X)\we\Drm^b(Y)$, then
    there exists an isometric automorphism $\sigma$ of $A\times_k\hat{A}$ such that
    $$\sigma\diamond\begin{pmatrix}
            X\\
            \hat{A}
        \end{pmatrix}=\begin{pmatrix}
            Y\\
            \hat{A}
        \end{pmatrix}$$
    in $\Hoh^1(k,A\times_k\hat{A})$. Conversely, if $k$ is perfect and if such an automorphism exists, then
    $\Drm^b(X)\we\Drm^b(Y,\beta)$ for some $\beta\in\Br(k)$.
    \begin{proof}
        Combining the
        exact sequence computing $\Br^A(k)$ and the medium-length exact sequence in
        nonabelian cohomology for the extension~\eqref{eq:gammaext}, we
        obtain the commutative diagram
        \begin{equation*}
            \xymatrix{
                                &                   & \Hoh^0(k,\ShAut_{\QC_{\dg}(A)}) \ar[d]\\
                                &                   & \Hoh^0(k,\Vbf^A) \ar[d]_\delta\\
                                &   \Hoh^1(k,A) \ar[d]_j \ar[r]                & \Hoh^1(k,A\times_k\hat{A}) \ar[d]_i &\\
                \Br(k) \ar[r]   &   \Br^A(k) \ar[r] & \Hoh^1(k,\ShAut_{\QC_{\dg}(A)}) \ar[d] \ar[r]  &  \Hoh^3(k,\Gm) \\
                                &                   & \Hoh^1(k,\Vbf^A)&
            }
        \end{equation*}
        where the row and column are exact, and where $j$ sends a homogeneous space $X$ for
        $A$ to the stack $\StQC_{\dg}^X$.
        The fibers of the map $i$ are precisely the orbits for the $\diamond$-action of
        $\Vbf^A(k)$ on $\Hoh^1_{\et}(k,A\times_k\hat{A})$
        by~\cite{serre-cohomologie}*{Proposition~39}.

        Suppose that $\Drm^b(X)\we\Drm^b(Y)$. Then, $\QC_{\dg}(X)\we\QC_{\dg}(Y)$
        by Proposition~\ref{prop:triadg}. Since $j(X)=j(Y)$, we see that $X$ and $Y$ are in
        the same $\Vbf^A(k)$-orbit in $\Hoh^1(k,A\times_k\hat{A})$. This proves the first
        part of the theorem.

        If $k$ is perfect, then the discussion before the theorem says that
        $\Vbf^A=\Ubf^A$. Hence, if there exists $\sigma\in\Ubf^A(k)=\Urm(A\times_k\hat{A})$
        such that $$\sigma\diamond\begin{pmatrix}
            X\\
            \hat{A}
        \end{pmatrix}=\begin{pmatrix}
            Y\\
            \hat{A}
        \end{pmatrix},$$
        then the images of $X$ and $Y$ in $\Hoh^1_{\et}(k,\ShAut_{\QC_{\dg}(A)})$
        coincide. It follows that $\StQC^{X}$ and $\StQC^{Y}$ are in the same
        $\Br(k)$-orbit in $\Br^A(k)$, which proves the second part of the theorem.
    \end{proof}
\end{theorem}

The theorem reduces entirely the problem of finding derived equivalences of principal
homogeneous spaces for $A$ to that of understanding the $\diamond$-action. In many interesting cases,
this can be accomplished by appealing to the special geometry of $A$. We include two
examples.

\begin{corollary}\label{cor:polarized}
    Let $A$ be an abelian variety over a field $k$ of characteristic $0$ such that
    $\End(A_{k^s})=\ZZ$ and such that the inclusion
    $\mathrm{NS}(A)\subseteq\mathrm{NS}(A_{k^s})=\ZZ$ is an equality. These
    conditions are satisfied by generic polarized abelian varieties. Suppose that $X$
    and $Y$ are principal homogeneous spaces for $A$. If $\Drm^b(X)\we\Drm^b(Y)$, then
    $aX=Y$ in $\Hoh^1(k,A)$ for some integer $a$ coprime to the order of $X$.
    \begin{proof}
        Since we work in characteristic $0$, $\Vbf^A=\Ubf^A$. Moreover, as explained
        in~\cite{orlov-abelian}*{Example 4.16},
        $\Ubf^A(k^s)=\Gamma_0(N)\subseteq\SL_2(\ZZ)$ for some integer $N$, where
        $\Gamma_0(N)$ is the congruence subgroup of matrices
        $\begin{pmatrix}a&b\\c&d\end{pmatrix}$ such that $c\equiv 0\,(\mathrm{mod}\,N)$. Our
        assumption that the Neron-Severi group of $A_{k^s}$ is defined over $k$
        implies that $\Ubf^A(k)=\Gamma_0(N)$ as well, and that
        $\ShAut_{\QC_{\dg}(A)}(l)\rightarrow\Ubf^A(l)$ is surjective for all finite
        extensions $l/k$. See for example the discussion in~\cite{huybrechts}*{Section 9.5}.
        Hence, the boundary map $\delta$ vanishes in diagram in the proof of
        Theorem~\ref{thm:twisted}.

        By our discussion preceding the theorem, the vanishing of $\delta$ implies that the $\diamond$-action of $\Gamma_0(N)\subseteq\SL_2(\ZZ)$ on
        $\Hoh^1_{\et}(k,A\times_k\hat{A})$ is the same as the action induced by
        $\Gamma_0(N)$ acting on the coefficient group.

        It follows that there exists a matrix
        $\begin{pmatrix} a  & b\\c&d\end{pmatrix}\in\Gamma_0(N)\subseteq\SL_2(\ZZ)$ such that
        \begin{equation*}
            \begin{pmatrix}
                a   &   b\\
                c   &   d
            \end{pmatrix}\cdot\begin{pmatrix}
                X\\
                \hat{A}
            \end{pmatrix}=\begin{pmatrix}
                aX\\
                c\phi_*X\end{pmatrix}=\begin{pmatrix}
                Y\\
                \hat{A}
            \end{pmatrix},
        \end{equation*}
        where $\phi$ is the polarization of $A$. Thus, $Y$ is
        in the subgroup of $\Hoh^1(k,A)$ generated by $X$. Since the argument is
        symmetric, it follows that $X$ and $Y$ generate the same subgroup. The theorem
        follows.
    \end{proof}
\end{corollary}

\begin{corollary}
    Let $A=E^n$ be the $n$-fold product of a geometrically non-CM elliptic curve $E$
    over a field $k$ of characteristic $0$. Suppose that $X$
    and $Y$ are principal homogeneous spaces for $A$. If $\Drm^b(X)\we\Drm^b(Y)$, then
    there is an integral symplectic matrix $\phi\in\mathrm{Sp}_{2n}(\ZZ)$ such that
    $$\phi\cdot\begin{pmatrix}
            X\\
            \hat{A}
        \end{pmatrix}=\begin{pmatrix}
            Y\\
            \hat{A}
        \end{pmatrix} $$
        in $\Hoh^1(k,A\times\hat{A})$.
    \begin{proof}
        The corollary follows as above from the fact that $\Vbf^A=\Ubf^A$ and that
        $\Ubf^A(k)=\Ubf^A(k^s)=\mathrm{Sp}_{2n}(\ZZ)$. See~\cite{orlov-abelian}*{Example
        4.15}.
    \end{proof}
\end{corollary}

\section{Derived equivalences of elliptic schemes}\label{sec:elliptic}

By an elliptic fibration, we mean a flat projective morphism $p:X\rightarrow S$ onto an
integral scheme $S$ such that the generic fiber $X_\eta$ is a smooth genus $1$ curve. If $p$ has a
section, we say that it is a Jacobian elliptic fibration.

Let $p:J\rightarrow S$ be a Jacobian elliptic fibration defined over the complex numbers with
smooth base $S$.
Then, there is a small resolution in the analytic category, resulting in a smooth complex
analytic manifold $\overline{J}\rightarrow J$. 
In~\cite{caldararu-elliptic}*{6.4}, C\u{a}ld\u{a}raru asks when
$\Drm^b(\overline{J},\alpha)\we\Drm^b(\overline{J},\beta)$ for two Brauer classes $\alpha$
and $\beta$ in $\Br(\overline{J})$. He found that if $J\rightarrow S$ is a so-called generic
elliptic $3$-fold, and if $\beta=a\alpha$, where $a$ is coprime to the order of $\alpha$,
then such an equivalence does exist. The definition of a generic elliptic $3$-fold is not
important for us, as we will work in far greater generality, and we give a partial answer to
C\u{a}ld\u{a}raru's question.

\begin{definition}
    Let $p:J\rightarrow S$ be a Jacobian elliptic fibration.
    We say that a triangulated equivalence $F:\Perf(J,\alpha)\we\Perf(J,\beta)$ is compatible
    with $p$ if $F$ is isomorphic to $\Phi_P$ for some
    $P\in\Perf(J\times_SJ,\alpha^{-1}\boxtimes\beta)$.
\end{definition}

Work of Canonaco and Stellari~\cite{canonaco-stellari} shows that if $J$ is smooth and projective over a field, then
every such triangulated equivalence $F$ is represented by a kernel $P\in\Drm^b(J\times_k
J,\alpha^{-1}\boxtimes\beta)$. The more important criterion is that $P$ be supported
scheme-theoretically on the closed subscheme $J\times_S J\subseteq J\times_k J$.

C\u{a}ld\u{a}raru's proof that $\Drm^b(\overline{J},\alpha)\we\Drm^b(\overline{J},\beta)$ as
above shows that
in fact the equivalence is compatible with the morphism $p$.
The equivalence is defined by a Fourier-Mukai transform defined by a specific
sheaf, in this case a universal sheaf for some moduli problem. Since the moduli problem is
relative to $S$, it is automatic that it is supported
not just on $\overline{J}\times_{\CC}\overline{J}$ but scheme-theoretically on
$\overline{J}\times_S\overline{J}$. The importance of this notion is encoded in the following
proposition.

\begin{proposition}
    Suppose that $F:\Perf(J,\alpha)\we\Perf(J,\beta)$ is compatible with $p$, say equal to
    $\Phi_P$ for some $P\in\Perf(J\times_SJ,\alpha^{-1}\boxtimes\beta)$. Then, for any map
    of schemes $T\rightarrow S$, $P$ restricts to a complex
    $P_T\in\Perf(J_T\times_TJ_T,\alpha^{-1}\boxtimes\beta)$ such that the induced map
    \begin{equation*}
        \Phi_{P_T}:\Perf(J_T,\alpha)\rightarrow\Perf(J_T,\beta).
    \end{equation*}
    is an equivalence.
    \begin{proof}
        By~\cite{bzfn}*{Theorem 1.2(2)},
        the hypothesis guarantees in fact that $F$ is the global section over $S$ of an
        equivalence of stacks $\StQC_{\dg}^{(J,\alpha)}\we\StQC_{\dg}^{(J,\beta)}$. The claim
        follows from this.
    \end{proof}
\end{proposition}

\begin{theorem}
    Suppose that $p:J\rightarrow S$ is a smooth Jacobian elliptic fibration, where $S$ is a connected
    regular noetherian scheme with characteristic $0$ field of fractions, and suppose that the geometric
    generic fiber $J_\eta$ is geometrically not CM. If there is an equivalence
    $F:\Drm^b(J,\alpha)\we\Drm^b(J,\beta)$ compatible with $p$ for some
    $\alpha,\beta\in\Br(J)$, then $\alpha$ and $\beta$ generate the same cyclic subgroup of
    $\Br(J)/\Br(S)$.
    \begin{proof}
        Under these conditions, there is an inclusion $\Br(J)\rightarrow\Br(J_\eta)$.
        By our hypotheses and the proposition, $F$ restricts to an equivalence
        $F_\eta:\Drm^b(J_\eta,\alpha)\we\Drm^b(J_\eta,\beta)$. The scheme $J_\eta$ is an
        elliptic curve over $\eta=\Spec k$, where $k$ is the field of fractions of $S$.
        There is an exact sequence
        \begin{equation*}
            0   \rightarrow \Br(k)  \rightarrow \Br(J_\eta) \rightarrow
            \Hoh^1(k,J_\eta)\rightarrow 0
        \end{equation*}
        since $J_\eta$ has a rational point. Moreover, this sequence is split for the same
        reason. So, we can write every class of $\Br(J_\eta)$ as $(X,\gamma)$ for $X$ in
        $\Hoh^1(k,J_\eta)$ and $\gamma\in\Br(k)$. Let $\alpha=(X,\gamma)$, and
        $\beta=(Y,\epsilon)$. Now, consider the commutative diagram
        \begin{equation*}
            \xymatrix{
                \Br(k)\ar[r]\ar[d]  &   \Br(J_\eta)\ar[d]&\\
                \Br(k)\ar[r]&           \Br^{J_\eta}(k)\ar[r]&
                \Hoh^1(k,\ShAut_{\QC_{\dg}(J_\eta)}),
            }
        \end{equation*}
        where the vertical map $\Br(J_\eta)\rightarrow\Br^{J_\eta}(k)$ sends a class
        $(X,\gamma)$ to $\Drm^b(X,\gamma)$. The action of $\sigma\in\Br(k)$ on $(X,\gamma)$
        is simply $\sigma\cdot(X,\gamma)=(X,\gamma+\sigma)$. It follows that the image of
        $(X,\gamma)$ in $\Hoh^1(k,\ShAut_{\QC_{\dg}}(J_\eta))$ is the same as $(X,0)$. In
        particular, $(X,0)$ and $(Y,0)$ have the same image. Now,
        the exact same argument as in the proof of
        Theorem~\ref{thm:twisted} shows that $aX=Y$ in $\Hoh^1(k,J_\eta)$
        for some integer $a$ prime to the order of $X$. Now,
        $a\alpha=(aX,a\gamma)=(Y,\epsilon+(a\gamma-\epsilon))=\beta+(0,a\gamma-\epsilon)$.
        As $a\alpha$ and $\beta$ are Brauer classes that are unramified over $J$, it follows that
        $a\gamma-\epsilon$ is too. But, this means that $a\gamma-\epsilon\in\Br(S)$.
        Since the argument is symmetric, this completes the proof.
% 
%         But, Theorem~\ref{mainappendix} now
%         implies that $\Drm^b(X)\we\Drm^b(Y)$, and hence that
%         $\Drm^b(X,\gamma)\we\Drm^b(Y,\gamma)$. Thus, $\Drm^b(Y,\gamma)\we\Drm^b(Y,\epsilon)$.
% 
%         If $X_\alpha$ is the image of $\alpha$ under this map, and if $X_\beta$ is the image
%         of $\beta$, then $\Drm^b(J_\eta,\alpha)\we\Drm^b(X_\alpha)$, while
%         $\Drm^b(J_\eta,\beta)\we\Drm^b(X_\beta)$. But, this implies that
%         $\Drm^b(X_\alpha)\we\Drm^b(X_\beta)$. The result now follows from
%         Theorem~\ref{mainappendix} or Corollary~\ref{cor:polarized}.
    \end{proof}
\end{theorem}

\begin{remark}
    The situation in~\cite{caldararu-elliptic} is more special in that the assumption is
    that the classes $\alpha$ and $\beta$ are of the form $(X,0)$ and $(Y,0)$ in the
    notation of the proof. For these classes, the same proof yields the stronger statement
    that $a\alpha=\beta$ for some $a$ coprime to the order of $\alpha$.
\end{remark}

\begin{remark}
    If the fibration is not smooth, but the rest of the hypotheses remain the same, then the
    same argument, using the compatibility of $F$ with $p$, shows that $\alpha$ and $\beta$
    generate the same subgroup of $\Br(J-D)/\Br(S-\Delta)$, where $\Delta\subseteq S$ is
    the (reduced) subscheme of points where $p$ is not smooth and $D=\Delta\times_S J$. In
    the situation where $\alpha=(X,0)$ and $\beta=(Y,0)$, we get that $a\alpha=\beta$ in
    $\Br(J-D)$ for $a$ coprime to the order of $\alpha$. But, since
    $\alpha,\beta\in\Br(J)\subseteq\Br(J-D)$, this is already true in $\Br(J)$.
\end{remark}

\begin{remark}
    One can also consider the special case of $\alpha=(X,\gamma)$ and $\beta=(X,0)$. In this
    case, we are asking about derived equivalences $\Drm^b(X,\gamma)\we\Drm^b(X)$ for
    $\gamma\in\Br(S)$. When $S=\Spec k$, Han~\cite{han} and Ciperiani-Krashen~\cite{ciperiani-krashen}
    have shown that the map $\Br(S)\rightarrow\Br(X)$ need not be injective.
    The elements $\gamma$ in the kernel of this map, the relative Brauer
    group $\Br(X/S)$, obviously give examples. In~\cite{antieau-tbs}*{Conjecture~2.13}, it
    is suggested that these are the only examples.
\end{remark}

The admittedly strong hypotheses of the theorem are satisfied in the examples in~\cite{caldararu-elliptic} and also for
example in Bridgeland and Maciocia's treatment~\cite{bridgeland-maciocia}*{Section 4} of derived
equivalences of elliptic surfaces of non-zero Kodaira dimension.

\end{document}